\documentclass[oneside]{amsart}%
\usepackage[latin9]{inputenc}
\usepackage{amsthm}
\usepackage{amstext}
\usepackage{amssymb}
\usepackage{esint}
\usepackage{amsfonts}%
\usepackage{amsmath}%
\usepackage{graphicx}
\providecommand{\U}[1]{\protect\rule{.1in}{.1in}}
\makeatletter
\numberwithin{equation}{section}
\numberwithin{figure}{section}
\theoremstyle{plain}
\newtheorem{thm}{\protect\theoremname}
\theoremstyle{plain}
\newtheorem{prop}[thm]{\protect\propositionname}
\theoremstyle{plain}
\newtheorem{lem}[thm]{\protect\lemmaname}
\theoremstyle{plain}
\newtheorem{cor}[thm]{\protect\corollaryname}
\theoremstyle{remark}
\newtheorem*{rem*}{\protect\remarkname}

\makeatother
  \providecommand{\corollaryname}{Corollary}
  \providecommand{\lemmaname}{Lemma}
  \providecommand{\propositionname}{Proposition}
  \providecommand{\remarkname}{Remark}
\providecommand{\theoremname}{Theorem}
\begin{document}
\title{On solutions to Cournot-Nash equilibria equations on the sphere}
\author{Micah Warren}
\thanks{I'm supported in part by  NSF DMS-0901644 and NSF DMS-1161498.  Thanks also to my representative Rush Holt for standing up to defend scientific funding.    }
\address{Department of Mathematics, Princeton University, Princeton NJ, USA}
\email{mww@math.princeton.edu, micahw@uoregon.edu}
\maketitle

\section{Introduction}

In this note, we discuss equations associated to Cournot-Nash Equilibria as
put forward by Blanchet and Carlier \cite{BC}. These equations are related to
an optimal transport problem in which the source measure is known but the
target measure is to be determined. A Cournot-Nash Equilibrium is a special
type of optimal transport: Each individual $x$ is transported to a point
$T(x)$ in a way that not only minimizes the total cost of transportation, but
minimizes a cost to the indivual $x$ (transportation plus other). \ This
latter cost may depend on the target distribution, and may involve congestion,
isolation and geographical terms.

Blanchet and Carlier demonstrated how CNE are related to nonlinear elliptic
PDEs, explicitly deriving a Euclidean version of the equation \cite[eq
4.6]{BC} and showing \cite[Theorem 3.8]{BC} that this problem has some very
nice properties. The fully nonlinear Monge-Amp\`ere equation differs from
`standard' optimal transport equations in that the potential itself occurs on
the right hand side, along with possibly some nonlocal terms. Here we study
the problem on the sphere. \ Immediately one can conclude from \cite[Theorem
3.8]{BC} and Loeper's \cite{MR2506751} results that that optimal maps are
continuous with control on the H\"older norm. We move this a step further and
show that all derivative norms can be controlled in terms of the data, when
the solution is smooth. \ When the solution is known to be differentiable
enough, then one can easily adapt Ma-Trudinger-Wang's \cite{MR2188047}
estimates. To make the conclusion a priori, we must use the continuity method.
Closedness follows Ma-Trudinger-Wang's estimates, but openness is not
immediate and requires some conditions. \ In Theorem \ref{thm:main} we give
some conditions on the data so that the problem can be solved smoothly. \ 

\section{Background and setup}

\ In this section we briefly recap the setup in \cite{BC}. \ Given a space of
player types $X$, endowed with a probability measure $\mu$, an action space
$Y,$ \ and a cost function
\[
\Phi:X\times Y\times\mathcal{P}(Y)\rightarrow\mathbb{R},
\]
$x$-type agents pay cost $\Phi(x,y,\nu)$ \ to take action $y$. Here $\nu
\in\mathcal{P}(Y)$ is the probability measure in the action space which is the
push forward of $\mu$ via by the map of actions from $X$ to $Y.$ Supposing
that $x$-type agents know the distribution $\nu,$ they can choose the best
action $y.$ \ A Cournot-Nash Equilibrium is a joint probability distribution
measure $\gamma\in$ $\mathcal{P}(X\times Y)$ with first marginal $\mu$ such
that
\begin{equation}
\gamma\left\{  (x,y)\in X\times Y:\Phi(x,y,\nu)=\min_{z\in Y}\Phi
(x,z,\nu)\right\}  =1\label{CNEcondition}%
\end{equation}
where $\nu$ is the second marginal. \ 

We will be interested in a particular type of cost
\[
\Phi(x,y,\nu)=c(x,y)+\mathcal{V}[\nu](y)
\]
where $c$ is the transportation cost. Blanchet and Carlier show \cite[Lemma
2.2]{BC} that a CNE will necessarily be an optimal transport pairing for the
cost $c$ between the measures $\mu$ and $\nu.$ They further show that if
$\mathcal{V}[\nu]$ is the differential of a functional $\mathcal{E}[\nu]$,
then at a minimizer for $\mathcal{E}[\nu]+\mathcal{W}_{c}(\mu,\nu),$ the
optimal transport will necessarily be a CNE. (Here, $\mathcal{W}_{c}(\mu,\nu)$
is Wasserstein distance.) In particular, if the cost $\mathcal{V}_{m}[\nu]$ is
of the form
\begin{equation}
\mathcal{V}_{m}[\nu](y)=f\left( \frac{d\nu}{dm}(y)\right) +\int\phi
(y,z)d\nu(z)+V(y)\label{VVV}%
\end{equation}
where $m$ is a `background' measure and the function $\phi(y,z)$ is symmetric
on $Y\times Y$, then $\mathcal{V}_{m}$ is a differential and solution to the
optimal transport is a CNE. (We will be licentious with notation, letting
$\nu$ denote not only the measure, but also the density with respect to the
background $m.$) From here on out we suppose we are working with a 
solution to an optimal transport with cost $c$ between measures $\mu$ and
$\nu$ which is also a CNE for a total cost $\Phi.$ We also assume that the manifolds $X$ and $Y$ are compact without boundary.  

One can consider the pair $(u,u^{\ast})$ which maximizes the Kantorovich
functional
\[
J(u,v)=\int-ud\mu+\int vd\nu
\]
over all $-u(x)+v(y)\leq\Phi(x,y).$ The pair $(u,u^{\ast})$ will satisfy
\begin{equation}
-u(x)+u^{\ast}(y)=\Phi(x,y)\label{eq:Kantorovich equality}%
\end{equation}
$\gamma$- almost everywhere, where $\gamma$ is the optimal measure for the
Kantorovich problem. If the cost satisfies the standard Spence-Mirrlees
condition (in the mathematics literature, the ``twist'' or (A1) condition
(c.f. \cite[section 2]{MR2188047})), we have $\mu$-almost everywhere
\begin{equation}
-u(x)+u^{\ast}(T(x))=\Phi(x,T(x)).\label{eq:MK equality}%
\end{equation}
The twist condition says that $T(x)$ is uniquely determined by
\begin{equation}
T(x)=\{y:Du(x)+Dc(x,y)=0\} ,\label{follows from twist-1}%
\end{equation}
which gives the identity
\begin{equation}
Du(x)+Dc(x,T(x))=0.\label{eq:twist identity}%
\end{equation}
Note that fixing an $x,$ the quantity
\[
\Phi(x,y)-u^{\ast}(y)
\]
must have a minimum at $T(x),$ we conclude that
\[
D_{y}\Phi(x,T(x))=Du^{\ast}(T(x)).
\]
But then we bring in the condition (\ref{CNEcondition}) that, fixing $x,$
\[
\Phi(x,T(x))\leq\Phi(x,y)
\]
which implies that 
\[
D_{y}\Phi(x,T(x))=0
\]
from which we conclude that
\[
Du^{\ast}(y)\equiv0.
\]
Now the pair $(u,u^{\ast})\ $is determined up to a constant. One can choose
the constant in $u$ or $u^{\ast}$ but not both. \ At this point we  simply
choose $u^{\ast}=0.$ Having fixed this choice allows us to read information about $u$ and
the measure $\nu,$ using (\ref{VVV}) and (\ref{eq:MK equality})
\[
-u(x)=c(x,T(x))+f(\nu(T(x)))+\int\phi(T(x),z)d\nu(z)+V(T(x)).
\]
In particular, the density $\nu(y)$ must be determined by
\begin{equation}
\nu(T(x))=f^{-1}\left\{  -u(x)-c(x,T(x))-\int\phi(T(x),T(z))d\mu(z)
(z)-V(T(x))\right\}  \label{eq:determines nu}%
\end{equation} 
having used the change of integration variables $T$ between $\mu$ and $\nu.$   
The optimal transportation equation (c.f. \cite{MR2188047}) becomes
\begin{equation}
\frac{\det(u_{ij}(x)+c_{ij}(x,T(x)))}{\det\left( -c_{is}(x,T(x))\right)
}=\frac{\mu(x)}{f^{-1}\left\{  Q(x,u)\right\}  }.\label{the equation}%
\end{equation}
Here and in the sequel, we use $i,j,k$ to denote derivatives in the source
$X$, and $p,s,t $ to denote derivatives in the target $Y$. \ It will
convenient to assume that $c_{is}$ is negative definite, which follows if we
are assuming the A2 condition (see \cite{MR2188047}) and have chosen an
appropriate coordinate system). We will use $b_{is}(x)=-c_{is}(x,T(x)).$
\ Also (to keep equations within one line) we abbreviate
\[
Q(x,u)=-u(x)-c(x,T(x))-\int\phi(T(x),T(z))d\mu(z)-V(T(x))
\]
with $T(x)$ being determined by (\ref{follows from twist-1}).

Before we say how this fully nonlinear equation is vulnerable, we
mention the ``Inada-like'' conditions \cite[Section 3.3]{BC} :
\begin{align}
\lim_{\nu\rightarrow0^{+}}f(\nu)  &  =-\infty\quad\mbox{and}\quad\lim
_{\nu\rightarrow+\infty}f(\nu)=+\infty\label{eq:inada}\\
&  f^{\prime}>0\quad\mbox{and}\quad f\in C^{2}(\mathbb{\mathbb{R}}%
^{+}).\label{eq:inada2}%
\end{align}
If $f$ satisfies these conditions, then several observations are in order.
\ First as noted in \cite[Theorem 3.8]{BC} on a compact manifold we get bounds
away from zero and infinity for the density $\nu.$ \ In the spherical distance
squared transportation cost case, this immediately gives $C^{\alpha}$
continuity of the map by results of Loeper. \ \ Secondly, the right hand side
of the equation (\ref{the equation}) is strictly monotone in the zeroth order
term - this is crucial in obtaining existence and uniqueness results, as it
will allow us to invert the linearized operator. \ \ Finally, as we will show
below, the first derivatives of this density will be bounded in terms of an a
priori constant (depending on smoothness of $f$) \ and the second derivatives
will be bounded by a constant times second derivatives of $u.$ \ These
estimates will allow us to take advantage of the Ma-Trudinger-Wang estimates.

We will show an estimate on smooth solutions: \ If a solution to
(\ref{the equation}) is $C^{4}$, then it enjoys estimates of all orders
subject to universal bounds. In order to show that arbitrary solutions are
$C^{4}$ and hence smooth, we must use a continuity method. This method relies
on a linearization which requires some discussion, given the integral terms in
the equation. \ \ 

The problem here, on a compact manifold, with cost function satisfying the
Ma-Trudinger-Wang condition, is quite a bit simpler than the more delicate
boundary value problem mentioned in \cite{BC}. With or without the nonlocal
terms, such a problem may be approached as in \cite{MR2644782}. We leave this
problem aside for now.

\section{Linearization}

We take the natural log of (\ref{the equation}) and then consider the
functional
\begin{alignat}{1}
\begin{aligned}F(x,u,Du,D^{2}u) & =\ln\det(\left(u_{ij}(x)+c_{ij}(x,T(x))\right)-\ln\det\left(b{}_{is}(x,T(x))\right)\end{aligned}\label{eq:definition of F}%
\\
-\ln\mu(x)+\ln f^{-1}\left( Q(x,u))\right) \nonumber
\end{alignat}
and the equation we want to solve is
\begin{equation}
F(x,u,Du,D^{2}u)=0.\label{lnequation}%
\end{equation}
\ Preparing for linearization, consider (\ref{eq:twist identity}) applied to
$u+tv$ :
\[
Du(x)+tD\eta(x)+Dc(x,T_{t}(x))=0.
\]
Differentiate with respect to $t$ and get that
\[
D\eta(x)=b_{is}(x,T(x))\frac{dT^{s}}{dt}.
\]
Linearizing,
\begin{align}
L\eta &  =\frac{d}{dt}F(u+t\eta)=w^{ij}\eta_{ij}+w^{ij}c_{ijs}b^{sk}\eta
_{k}+b^{is}c_{isp}b^{pk}\eta_{k}\label{eq:L0}\\
+  &  \frac{\left( f^{-1}(Q)\right) ^{\prime}}{f^{-1}\left( Q\right) }\left\{
\begin{array}
[c]{c}%
-c_{s}(x,T(x))b^{sk}\eta_{k}-\eta-V_{s}b^{sk}\eta_{k}\\
-b^{sk}\eta_{k}(x)\int\phi_{s}(T(x),T(z))d\mu(z)-\int\phi_{\bar{s}%
}(T(x),T(z))b^{sk}(z)\eta_{k}(z)d\mu(z)
\end{array}
\right\} .\label{eq:L1}%
\end{align}
Here we are using
\[
w_{ij}(x)=u_{ij}(x)+c_{ij}(x,T(x)).
\]
We note also that differentiating (\ref{eq:twist identity}) shows
\begin{equation}
T_{i}^{s}(x,T(x))=\frac{\partial T^{s}}{\partial x_{i}}=b^{sk}(x,T(x))w_{ki}%
(x,T(x)).\label{nine}%
\end{equation}
Splitting (\ref{eq:L0}) and (\ref{eq:L1}) for convenience we write,
respectively,
\[
L\eta=L^{0}\eta+L^{1}\eta.
\]
We take $g_{ij}(x)=w_{ij}(x)$ to define a metric (one can check that it
transforms as such), then write
\begin{equation}
d\mu(x)=e^{-a(x)}dV_{g}(x)
\end{equation}
where
\[
-a(x)=\ln\mu(x)-\frac{1}{2}\ln\det w_{ij}(x).
\]
From the definition of $F$ (\ref{eq:definition of F}) we have
\[
-a(x)=\frac{1}{2}\ln\det w_{ij}-\ln\det b+\ln\nu-F,
\]
having introduced
\[
\nu(x)=\ln f^{-1}\left( Q(x,u))\right) .
\]
First, we compute the Bakry-Emery Laplace
\[
\triangle_{a}\eta=\bigtriangleup_{g}\eta-\nabla a\cdot\nabla\eta.
\]
We begin with $\bigtriangleup_{g}\eta$ differentiating in some coordinate
system (see very similar computations preceding  \ref{pg10}):
\[
\frac{1}{\sqrt{\det w}}\left( \sqrt{\det w}w^{ij}\eta_{j}\right) _{i}%
=w^{ij}\eta_{ij}+\frac{1}{2}w^{ab}\partial_{i}w_{ab}w^{ij}\eta_{j}%
-w^{ia}w^{bj}\partial_{i}w_{ab}\eta_{j}
\]
\begin{align*}
&  =w^{ij}\eta_{ij}+w^{ab}w^{ij}\left( \partial{}_{i}w_{ab}-\partial_{b}%
w_{ia}\right) \eta_{j}-\frac{1}{2}w^{ab}\partial_{i}w_{ab}w^{ij}\eta_{j}\\
&  =w^{ij}\eta_{ij}+\left( w^{ba}c_{abs}b^{sj}-w^{ij}c_{isk}b^{sk}\right)
\eta_{j}-\frac{1}{2}w^{ij}\left( \ln\det w\right) _{i}\eta_{j}%
\end{align*}
\[
=L^{0}\eta-b^{is}c_{isp}b^{pk}\eta_{k}-w^{ij}c_{kis}b^{sk}\eta_{j}-\frac{1}%
{2}w^{ij}\left( \ln\det w\right) _{i}\eta_{j}.
\]
Thus
\begin{align*}
\triangle_{a}\eta &  =L^{0}\eta-b^{is}c_{isp}b^{pk}\eta_{k}-w^{ij}%
c_{kis}b^{sk}\eta_{j}-\frac{1}{2}w^{ij}\left( \ln\det w\right) _{i}\eta_{j}\\
&  +\frac{1}{2}w^{ij}\left( \ln\det w\right) _{i}\eta_{j}-w^{ij}\left( \ln\det
b\right) _{i}\eta_{j}+\left( \ln\nu\right) _{i}w^{ij}\eta_{j}-F_{i}w^{ij}%
\eta_{j}\\
&  =L^{0}v+\left( \ln\nu\right) _{i}w^{ij}\eta_{j}-F_{i}w^{ij}\eta_{j},
\end{align*}
and hence
\[
L\eta=\triangle_{a}\eta+L^{1}\eta-(\ln \nu)_{i}w^{ij}\eta_{j}+F_{i}w^{ij}%
\eta_{j}.
\]
Next, we compute
\[
(\ln \nu)_{i}=\frac{\left( f^{-1}(Q)\right) ^{\prime}}{f^{-1}\left( Q\right)
}\left\{
\begin{array}
[c]{c}%
-u_{i}(x)-c_{i}(x,T(x))-c_{s}(x,T(x))b^{sk}w_{ki}\\
-b^{sk}w_{ki}\int\phi_{s}(T(x),T(z))d\mu(z)-V_{s}b^{sk}w_{ki}%
\end{array}
\right\}  .
\]
Noting that $-u_{i}(x)-c_{i}(x,T(x))=0,$ and the expression (\ref{eq:L1}) we
have
\begin{align*}
L^{1}\eta-(\ln \nu)_{i}w^{ij}\eta_{j}  &  =\\
&  \frac{\left( f^{-1}(Q)\right) ^{\prime}}{f^{-1}\left( Q\right) }\left\{
-\eta-\int\phi_{{s}}(T(x),T(z))b^{sk}(z)\eta_{k}(z)d\mu(z)\right\}  .
\end{align*}
Next, we compute the integral term in the previous expression: Notice
\[
\int\left\langle {\nabla}\phi(y,T(z)),\nabla \eta \right\rangle e^{-a(z)}%
dV_{g}(z)=\int\phi_{{s}}(y,T(z))b^{sk}w_{ki}\eta_{j}w^{ij}e^{-a(x)}dV_{g}
\]
\[
=\int\phi_{{s}}(T(x),T(z))b^{sk}\eta_{k}(z)d\mu(z).
\]
Now, integrating by parts, we have that
\[
-\int\phi_{{s}}(T(x),T(z))b^{sk}\eta_{k}(z)d\mu(z)=\int\phi
(T(x),T(z))\bigtriangleup_{a}\eta(z)e^{-a(z)}dV_{g}(z).
\]
Combining, we have
\begin{equation}
L\eta=\bigtriangleup_{a}\eta-h(x)\eta(x)-h(x)\int\phi(T(x),T(z))\bigtriangleup
_{a}\eta(z)d\mu(z)+\left\langle \nabla F,\nabla\eta\right\rangle
,\label{full linearized operator}%
\end{equation}
using the shorthand
\[
h(x,u)=\frac{\left( f^{-1}(Q)\right) ^{\prime}}{f^{-1}\left( Q\right) }.
\]
Note here that if $f\:$satisfies (\ref{eq:inada}),(\ref{eq:inada2}) then
$h(x,u)$ will be a positive differentiable quantity. In particular, if $f(\tau)=ln(\tau)$
then $h$ will be identically $1.$ When $F\equiv0$ we have the following.

\begin{prop}
At a solution of (\ref{lnequation}), the linearized operator takes the form
\begin{equation}
L\eta=\bigtriangleup_{a}\eta-h(x)\eta(x)-h(x)\int\phi(T(x),T(z))\bigtriangleup
_{a}\eta(z)d\mu(z).\label{definition of linearized operator at a solution}%
\end{equation}

\end{prop}

\begin{lem}
\label{lem:invertible L}Suppose that
\begin{equation}
\max_{\left( x,y\right) \in X\times Y}h(x,u)|\phi
(x,y)|<1.\label{eq:condition for invert}%
\end{equation}
Then the operator (\ref{definition of linearized operator at a solution}%
)\textup{ }has trivial kernel.
\end{lem}

\begin{proof}
To make use of some functional analytic formality, we define operators $A,J,h$
and $I$ on the space
\[
\mathcal{{B}}=L^{2}(X,d\mu)
\]
via
\begin{align*}
[A\eta](x)  &  =\bigtriangleup_{a}\eta(x),\\
{}[J\eta](x)  &  =\int\phi(T(x),T(z))\eta(z)d\mu(z),\\
{}[h\eta](x)  &  =h(x)\eta(x)\\
{}[I\eta](x)  &  =\eta(x).
\end{align*}
Then
\[
L=A-h-hJA=\left( I-hJ\right) A-h=(I-hJ)\left( A-\left( I-hJ\right)
^{-1}h\right) .
\]

First we have the pointwise estimate%

\begin{align*}
&  [hJ\eta](x) = \int h(x)\phi(T(x),T(y))\eta(y)d\mu(y)\\
&  \leq\left\Vert \int h(x)\phi(T(x),T(y))d\mu(x)\right\Vert _{L^{2}}%
^{1/2}\left\Vert \eta\right\Vert _{L^{2}}^{1/2}\\
&  \leq\left[ \max_{x,y\in X \times Y }h(x)|\phi(x,y)|\right] ^{1/2}<\left\Vert
\eta\right\Vert _{L^{2}}^{1/2}%
\end{align*}

using (\ref{eq:condition for invert}).  Integrating this quantity over $\mu$
yields that
\[
\left\Vert hJ\right\Vert < 1
\]
as an operator on $\mathcal{{B}}$. It then makes sense to talk about $\left(
I-hJ \right) ^{-1}$ . Thus
\[
Ker(L)=Ker\left( A-\left( I-hJ\right) ^{-1}h\right) .
\]
Now suppose for purposes of contradiction, that we have nontrivial $\eta\in
Ker(L)$. Then
\[
A\eta=\left( I-hJ\right) ^{-1}h\eta
\]
thus
\[
\left\langle \left( I-hJ\right) ^{-1}h\eta,\eta\right\rangle =\left\langle
A\eta,\eta\right\rangle =-\int\left\vert \nabla\eta\right\vert ^{2}d\mu<0.
\]
But as $\left( I-hJ\right) $ is invertible we can let
\[
\left( I-hJ\right) \omega=h\eta
\]
that is
\[
\left\langle \omega,h^{-1}\left( I-hJ\right) \omega\right\rangle =\left\langle
\left( I-hJ\right) ^{-1}h\eta,\eta\right\rangle <0
\]
that is
\[
0>\left\langle \omega,\frac{1}{h}\omega\right\rangle -\left\langle
\omega,J\omega\right\rangle \geq\frac{1}{\max h}\left\Vert \omega\right\Vert
^{2}-\left\Vert J\right\Vert \left\Vert \omega\right\Vert ^{2}=\left( \frac
{1}{\max h}-\left\Vert J\right\Vert \right) \left\Vert \omega\right\Vert ^{2}
\]
which is clearly a contradiction if $1>\max h\left\Vert J\right\Vert .$
\end{proof}

\section{Estimates on the sphere}

From here out we specialize to the round unit sphere, with cost function half of
distance squared. Note that this sphere has Riemannian volume $n\omega_n.$  

\subsection*{Oscillation estimates}

The following estimates are a version of \cite[Lemma 3.7]{BC}. On a compact
manifold, the cost function will be bounded. Since the solution $u$ is
$c$-convex, at the maximum point $x_{max}$ of $u$, $u$ is supported below by
cost support function $c(x,T(x_{0}))+\lambda.$ Hence, at the minimum point
$x_{min}$ \ we have that $u(x_{min})\geq$$c(x_{min},T(x_{max}))+\lambda$,
which in turn tells us that
\[
\text{{osc}}u\leq\text{{osc}}\,c=\frac{\pi^{2}}{2}.
\]
Next we observe that, because integration of the density $\nu$ against $m$
gives a probability measure, the density $\nu$ must be larger than $1/n\omega_n$ at some
point $y_{0}.$ It follows that, at the point $x{}_{0}=T^{-1}(y_{0})$ using
(\ref{eq:determines nu})
\[
-c(x_{0},y_{0})-u(x_{0})-\int\phi(y_{0},T(z))d\mu(z)-V(y_{0})\geq f(\frac{1}{n\omega_n})
\]
and similarly at the point where the density $\nu$ is smallest, $x_{1}$
\[
-c(x_{1},y_{1})-u(x_{1})-\int\phi(y_{1},T(z))d\mu(z)-V(y_{1})=f(\nu(x_{1}))
\]
Hence,
\begin{multline*}
-c(x_{0},y_{0})+c(x_{1},y_{1})-u(x_{0})+u(x_{1})-\int\left( \phi
(y_{0},T(z))+\phi(y_{1},T(z))\right) d\mu(z)-V(y_{0})+V(y_{1})\\
\geq f(\frac{1}{n\omega_n})-f(\nu(x_{1}))
\end{multline*}
that is
\[
f(\nu(x_{1}))\geq f(\frac{1}{n\omega_n})-2\text{{osc}}\,c-2\text{{osc}}\,\phi-\text{{osc}%
}\,V>-\infty.
\]
By Inada's condition,
\[
\nu\geq f^{-1}\left( f(\frac{1}{n\omega_n})-\pi^{2}-2\text{{osc}}\,\phi-\text{{osc}}\,V\right)
>0.
\]
Similarly, an upper bound can be derived
\[
\nu\leq f^{-1}\left( f(\frac{1}{n\omega_n})+\pi^{2}+2\text{{osc}}\,\phi+\text{{osc}}\,V\right)
<\infty.
\]

\subsection{Stayaway}

Now that $\nu$ is under control, it follows from the stayaway estimates of
Delano\"e and Loeper \cite{MR2232207} that the map $T(x)$ must satisfy
\[
dist_{\mathbb{S}^{n}}(x,T(x))\leq\pi-\epsilon(f,\mu,V,\phi)
\]
In particular the map stays clear of the cut locus. All derivatives of the
cost function are now controlled.

\subsection*{MTW estimates}

\begin{lem}
\label{lem:nubounded}If the map $T$ is differentiable and locally invertible,
then the target measure density
\[
\nu(T(x))=f^{-1}\left( -c(x,T(x))-u(x)-\int\phi(T(x),T(z))d\mu
(z)-V(T(x))\right)
\]
has first derivatives bounded by a universal constant and has second
derivatives bounded as
\[
\nu_{sr}=C_{1}+C_{2k}\left( T^{-1}\right) _{r}^{k}
\]
where the constants are within a controlled range.
\end{lem}

\begin{proof}
Differentiate in the $x_{k}$ direction
\[
\nu_{s}T_{k}^{s}(x)=
\]
\[
\left( f^{-1}\right) \prime\left\{  -c_{k}(x,T(x))-c_{s}(x,T(x))T_{k}%
^{s}-u_{k}-T_{k}^{s}\int\phi_{s}(T(x),T(z))d\mu(z)-V_{s}T_{k}^{s}\right\}
\]

\[
=\left( f^{-1}\right) \prime T_{k}^{s}(x)\left\{  -c_{s}(x,T(x))-\int\phi
_{s}(T(x),T(z))d\mu(z)-V_{s}(T(x))\right\}  .
\]
As this is true for all $k$ and $DT$ is invertible, we can conclude that
\[
\nu_{s}(T(x))=\left( f^{-1}\right) \prime\left\{  -c_{s}(x,T(x))-\int\phi
_{s}(T(x),T(z))d\mu(z)-V_{s}(T(x))\right\}  ,
\]
which is a bounded quantity. For second derivatives, differentiate this
equation in $x$ again
\begin{multline*}
\nu_{sp}T_{k}^{p}=\\
\left( f^{-1}\right) \prime\prime T_{k}^{p}(x)\times\\
\left\{  -c_{s}(x,T(x))-\int\phi_{s}(T(x),T(z))d\mu(z)-V_{s}(T(x))\right\} \\
\times\left\{  -c_{p}(x,T(x))-\int\phi_{p}(T(x),T(z))d\mu(z)-V_{p}%
(T(x))\right\} \\
+\left( f^{-1}\right) \prime\left\{  -c_{sk}(x,T(x))-c_{sp}(x,T(x))T_{k}%
^{p}(x) -T_{k}^{p}(x)\int\phi_{ps}(T(x),T(z))d\mu(z)-T_{k}^{p}(x)V_{sp}%
(T(x))\right\} \\
\end{multline*}
that is
\begin{multline*}
\nu_{sr}=\left( f^{-1}\right) \prime\prime\times\\
\left\{  -c_{s}(x,T(x))-\int\phi_{s}(T(x),T(z))d\mu(z)-V_{s}(T(x))\right\} \\
\times\left\{  -c_{p}(x,T(x))-\int\phi_{p}(T(x),T(z))d\mu(z)-V_{p}%
(T(x))\right\} \\
+\left( f^{-1}\right) \prime\left\{  -c_{sk}(x,T(x))\left( T^{-1}\right)
_{r}^{k}-c_{sp}(x,T(x))-\int\phi_{ps}(T(x),T(z))d\mu(z)-V_{sp}(T(x))\right\}
.
\end{multline*}
Now all the terms, with the exception of the $\left( T^{-1}\right) _{r}^{k}$
term, are in given by controlled constants, independent of $u.$ We are done.
\end{proof}

Before we state the main a priori estimate, we recall the Ma-Trudinger-Wang
tensor \cite[pg. 154]{MR2188047}. For each $y$ in the target, one can define
Ma-Trudinger-Wang (MTW) tensor as a $\left( 2,2\right) $ tensor on $T_{x}M$ via%

\[
\mathsf{MTW}_{ij}^{kl}(x,y)=\left\{  \left(  -c_{ijpr}+c_{ijs}c^{sm}%
c_{mrp}\right)  c^{pk}c^{rl}\right\}  (x,y).
\]
It is by now a well known fact that, on the sphere
\[
\mathsf{MTW}_{ij}^{kl}\xi_{k}\xi_{l}\tau^{i}\tau^{j}\geq\delta_{n}\left\Vert
\xi\right\Vert ^{2}\left\Vert \tau\right\Vert ^{2}%
\]
for a positive $\delta_{n}$ and all vector-covector pairs such that
\[
\xi(\tau)=0.
\]
(For a more discussion of the geometry of this tensor, see also
\cite{MR2654086}.)

\bigskip Given a solution, we define an operator on $(2,0)$
tensors as follows. \  Let $h$ be a $(2,0)$ tensor. \ Given vector fields
$X_{1},X_{2}$, we define
\[
\left(  L_{w}h\right)  (X_{1},X_{2})=\frac{1}{\sqrt{\det w}}\nabla_{j}\left(
\sqrt{\det w}w^{ij}\nabla_{i}h\right)  -w^{ij}\nabla_{j}a\nabla_{i}%
h(X_{1},X_{2})
\]
where
\[
-a(x)=\frac{1}{2}\ln\det w(x)-\ln\det b(x)+\ln\nu(x,T(x))
\]
and covariant differention is taken with respect to the round metric.

\begin{prop}
\bigskip Let $u$ be a solution of (\ref{the equation}). \ If $e$ is a unit direction in a
local chart on $S^{n}$ then \
\begin{align*}
L_{w}w(e,e) &  \geq w^{ij}(-c_{ijpr}+c_{ijs}c_{krp}c^{sk})c^{pm}c^{rl}%
w_{me}w_{le}\\
&  -C\left(  1+\sum w^{ii}\sum w_{jj}+\sum w^{ii}+\sum w_{ii}^{2}\right)
\end{align*}

\end{prop}

\begin{proof}
This was proven by Ma Trudinger and Wang in \cite{MR2188047}, \ in the case where
densities are known ahead of time. \ \ Adapting their proof requires only a small modification somewhere in the middle, but for completeness (and
mostly for fun), we will present the calculation. \ 

First, we note that
\begin{align*}
\left(  \frac{\partial_{j}\left(  \sqrt{\det w}w^{ij}\right)  }{\sqrt{\det
w_{ij}}}-w^{ij}a_{j}\right) = 
 \partial_{j}w^{ij}+\frac{1}{2}w^{ij}\left(
\ln\det w\right)  _{j}+  w^{ij}\frac{1}{2}\left(  \ln\det w\right)  _{j}%
-w^{ij}\left(  \ln\det b\right)  _{j}+w^{ij}\left(  \ln\nu\right)  _{s}%
T_{j}^{s} \\
 =-w^{ia}w^{bj}\partial_{j}w_{ab}+w^{ij}\left(  \ln\det w\right)
_{j}-w^{ij}\left(  b^{sk}b_{skj}+b^{sk}b_{skt}T_{j}^{t}\right)  +b^{si}\left(
\ln\nu\right)  _{s}   \\  
 =-w^{ia}w^{bj}(\partial_{j}w_{ab}-\partial_{a}w_{bj})-w^{ia}w^{bj}%
\partial_{a}w_{bj}+w^{ij}\left(  \ln\det w\right)  _{j}-w^{ij}b^{sk}%
b_{skj}-b^{ti}b^{sk}b_{skt}+b^{si}\left(  \ln\nu\right)  _{s}%
\end{align*}%
\begin{align}
&  =-w^{ia}w^{bj}(c_{abs}T_{j}^{s}-c_{bjs}T_{a}^{s})-w^{ij}b^{sk}%
b_{skj}-b^{ti}b^{sk}b_{skt}+b^{si}\left(  \ln\nu\right)  _{s} \label{pg10} \\
&  =b^{si}w^{bj}c_{bjs}-b^{ti}b^{sk}b_{skt}+b^{si}\left(  \ln\nu\right)  _{s}%
\end{align}
using (among others)\ the relations
\begin{align*}
\left(  \partial_{j}w_{ab}-\partial_{a}w_{bj}\right)    & =c_{abs}T_{j}%
^{s}-c_{bjs}T_{a}^{s}\\
w^{bj}T_{j}^{s}  & =b^{sj}.
\end{align*}
Now%
\begin{align*}
L_{w}w(e_{1},e_{1}) &  =\frac{1}{\sqrt{\det w}}\nabla_{j}\left(  \sqrt{\det
w}w^{ij}\nabla_{i}w\right)  (e_{1},e_{1})-w^{ij}\nabla_{j}a\nabla_{i}%
w(e_{1},e_{1})\\
&  =w^{ij}\nabla_{j}\nabla_{i}w(e_{1},e_{1})+\left(  b^{si}w^{bj}%
c_{bjs}-b^{ti}b^{sk}b_{skt}+b^{si}\left(  \ln\nu\right)  _{s}\right)
\nabla_{i}w(e_{1},e_{1})\\
&  =w^{ij}\left(
\begin{array}
[c]{c}%
\partial_{i}\partial_{j}w(e_{1},e_{1})-\nabla_{j}\partial_{i}w(e_{1}%
,e_{1})+2w(\nabla_{\nabla_{j}\partial_{i}}e_{1},e_{1})\\
2\partial_{i}w(\nabla_{j}e_{1},e_{1})-2\partial_{j}w(\nabla_{i}e_{1},e_{1})\\
+2w(\nabla_{j}\nabla_{i}e_{1},e_{1})+2w(\nabla_{i}e_{1},\nabla_{j}e_{1})
\end{array}
\right)  \\
&  +\left(  b^{si}w^{bj}c_{bjs}-b^{ti}b^{sk}b_{skt}+b^{si}\left(  \ln
\nu\right)  _{s}\right)  \left(  \partial_{i}w(e_{1},e_{1})-2w(\nabla_{i}%
e_{1},e_{1})\right)  .
\end{align*}
At this point, we choose a normal coordinate system (in the round metric), and
we have
\begin{align*}
L_{w}w(e_{1},e_{1}) &  =\left(  b^{si}w^{bj}c_{bjs}-b^{ti}b^{sk}b_{skt}%
+b^{si}\left(  \ln\nu\right)  _{s}\right)  \partial_{i}w(e_{1},e_{1}%
)+w^{ij}\left(  \partial_{i}\partial_{j}w(e_{1},e_{1})+2w(\nabla_{j}\nabla
_{i}e_{1},e_{1})\right)  \\
&  =\left(  b^{is}w^{bj}c_{bjs}-b^{it}b^{sk}b_{skt}+b^{is}\left(  \ln
\nu\right)  _{s}\right)  \partial_{i}w_{11}\\
&  +w^{ij}\left(  \partial_{i}\partial_{j}w_{11}-\partial_{1}\partial
_{1}w_{ij}\right)  +w^{ij}\left(  \partial_{1}\partial_{1}w_{ij}+2w(\nabla
_{j}\nabla_{i}e_{1},e_{1})\right)
\end{align*}

Again harking back to \cite{MR2188047}, we let
\[
K=C\sum w^{ii}\sum w_{jj}+C\sum w^{ii}+C\sum w_{ii}^{2}+C
\]
and note that terms of the following form are $K$
\begin{align*}
K &  =w^{ij}T_{b}^{s}\\
K &  =\left(  \partial_{j}w_{ik}-\partial_{k}w_{ij}\right)  \\
K &  =w^{ij}2w(\nabla_{j}\nabla_{i}e_{1},e_{1})\\
K &  =w^{ij}w_{kl}%
\end{align*}
so that
\begin{align*}
L_{w}w(e_{1},e_{1}) &  =-K+\left(  b^{si}w^{bj}c_{bjs}-b^{ti}b^{sk}%
b_{skt}+b^{si}\left(  \ln\nu\right)  _{s}\right)  \partial_{i}w_{11}\\
&  +w^{ij}\left(  \partial_{i}\partial_{j}w_{11}-\partial_{1}\partial
_{1}w_{ij}\right)  +w^{ij}\partial_{1}\partial_{1}w_{ij}.
\end{align*}
Now differentiating
\begin{equation}
\ln\det w_{ij}=\ln\det b_{is}+\ln\mu-\ln\nu
\end{equation}
we have
\begin{equation}
w^{ij}\partial_{1}w_{ij}=b^{si}\left(  b_{is1}+b_{ist}T_{1}^{t}\right)
+\left(  \ln\mu\right)  _{1}-\left(  \ln\nu\right)  _{s}T_{1}^{s}%
\label{detder}%
\end{equation}
and again%
\[
w^{ij}\partial_{11}w_{ij}+\partial_{1}w^{ij}\partial_{1}w_{ij}=K+b^{si}%
b_{ist}T_{11}^{t}+\left(  \ln\nu\right)  _{sr}T_{1}^{r}T_{1}^{s}-\left(
\ln\nu\right)  _{s}T_{11}^{s}.
\]
Now recall Lemma 3,
\begin{align*}
\left(  \ln\nu\right)  _{sr}T_{1}^{r}T_{1}^{s} &  =\frac{C_{1sr}%
+C_{2sk}\left(  T^{-1}\right)  _{r}^{k}}{\nu}T_{1}^{r}T_{1}^{s}-\left(  \ln
\nu\right)  _{s}\left(  \ln\nu\right)  _{r}T_{1}^{r}T_{1}^{s}\\
&  =K
\end{align*}
thus%
\begin{equation}
w^{ij}\partial_{11}w_{ij}=w^{ia}w^{bj}\partial_{1}w_{ab}\partial_{1}%
w_{ij}+K+b^{si}b_{ist}T_{11}^{t}-\left(  \ln\nu\right)  _{s}T_{11}%
^{s}.\label{concavity}%
\end{equation}
Note that differentiating
\[
T_{i}^{s}=b^{sk}w_{ki}%
\]
yields
\begin{equation}
T_{ij}^{s}=b^{sk}\partial_{j}w_{ki}-b^{sa}b^{pk}w_{ki}\left(  b_{apj}%
+b_{apq}T_{j}^{q}\right)  \label{Tsij}%
\end{equation}
in particular
\[
T_{11}^{s}=b^{sk}\partial_{1}w_{k1}-b^{sa}b^{pk}w_{k1}\left(  b_{ap1}%
+b_{apq}T_{1}^{q}\right)  .
\]
Now it follows that
\begin{align}
&  T_{11}^{s}-b^{sk}\partial_{k}w_{11}\nonumber\\
&  =b^{sk}\left(  \partial_{1}w_{k1}-\partial_{k}w_{11}\right)  -b^{sa}%
b^{pk}w_{k1}\left(  b_{ap1}+b_{apq}T_{1}^{q}\right)  \label{comm2}\\
&  =K.\label{comm3}%
\end{align}
Bringing in concavity of the Monge-Amp\`ere (\ref{concavity}) and (\ref{comm3}%
)\ we can eliminate some terms to see
\begin{align*}
L_{w}w(e_{1},e_{1}) &  \geq-K+b^{is}w^{bj}c_{bjs}\partial_{i}w_{11}\\
&  +w^{ij}\left(  \partial_{i}\partial_{j}w_{11}-\partial_{1}\partial
_{1}w_{ij}\right)  .
\end{align*}
Then using
\begin{align*}
\partial_{1}\partial_{1}w_{ij} &  =u_{ij11}+c_{ij11}+2c_{ijs1}T_{1}%
^{s}+c_{ijs}T_{11}^{s}+c_{ijpr}T_{1}^{p}T_{1}^{r}\\
\partial_{i}\partial_{j}w_{11} &  =u_{11ij}+c_{11ij}+c_{11si}T_{j}%
^{s}+c_{11sj}T_{i}^{s}+c_{11s}T_{ij}^{s}+c_{11pr}T_{i}^{p}T_{j}^{r}%
\end{align*}
we have
\begin{align*}
L_{w}w(e_{1},e_{1}) &  \geq-K+\left(  b^{is}w^{bj}c_{bjs}\right)  \partial
_{i}w_{11}\\
&  +w^{ij}\left(  c_{11s}T_{ij}^{s}+c_{11pr}T_{i}^{p}T_{j}^{r}-c_{ijs}%
T_{11}^{s}-c_{ijpr}T_{1}^{p}T_{1}^{r}\right).
\end{align*}
\ From  (\ref{Tsij})
\begin{align*}
w^{ij}T_{ij}^{s} &  =w^{ij}\left(  b^{sk}\partial_{j}w_{ki}-b^{sa}b^{pk}%
w_{ki}\left(  b_{apj}+b_{apq}T_{j}^{q}\right)  \right)  \\
=w^{ij}b^{sk}\left(  \partial_{j}w_{ki}-\partial_{k}w_{ij}+\partial_{k}%
w_{ij}\right)   &  -b^{sa}b^{pj}\left(  b_{apj}+b_{apq}T_{j}^{q}\right)
\end{align*}

\begin{align*}
&  =K+b^{sk}\partial_{k}\left(  \ln\det w\right)  \\
&  =K
\end{align*}
by (\ref{detder}). \ \ Using   (\ref{comm2}) we conclude
\begin{align*}
L_{w}w(e_{1},e_{1}) &  \geq-K-w^{bj}c_{bjs}b^{sa}b^{pk}w_{k1}b_{apq}T_{1}%
^{q}\\
&  -w^{ij}c_{ijpr}T_{1}^{p}T_{1}^{r}.
\end{align*}
which is the desired result after reindexing. \ \

\end{proof}

\begin{cor}
Second derivatives of $u$ are uniformly bounded.
\end{cor}

\begin{proof}
Given the maximum principle estimate, this proof is standard following
\cite{MR2188047}. For some more details in the setting of Riemannian manifolds see \cite[Theorem 3.5]{Kim03102011}.
\end{proof}

\section{Main theorem}

In order to make a precise statement, we define
\begin{align*}
\nu_{lower}  &  =f^{-1}\left( f(\frac{1}{n\omega_n})-2\text{{osc}}c-2\left\Vert \phi\right\Vert
_{\infty}-\text{{osc}}V\right) \\
\nu_{upper}  &  =f^{-1}\left( f(\frac{1}{n\omega_n})+2\text{{osc}}c+2\left\Vert \phi\right\Vert
_{\infty}+\text{{osc}}V\right).
\end{align*}
Similarly, an upper bound can be defined
\[
h_{\max}=\sup_{Q\in\lbrack\nu_{lower},\nu_{upper}]}\frac{\left( f^{-1}%
(Q)\right) ^{\prime}}{f^{-1}\left( Q\right) }.
\]

\begin{thm}
\label{thm:main} Suppose that $f$ satisfies the Inada-like conditions
(\ref{eq:inada}) (\ref{eq:inada2}), $\mu$ and $m$ are smooth, and $\phi$ and
$V$ are lipschitz. \ If
\begin{equation}
\max_{x,y\in M}|\phi(x,y)|<\frac{1}{h_{\max}}.
\end{equation}
then there exists a smooth solution to (\ref{lnequation}).
\end{thm}

For existence, we proceed by continuity \cite[Theorem 17.6]{GT} on the
equation (\ref{lnequation}), letting
\begin{multline}
F(t,x,u,Du,D^{2}u)=\\
\ln\det(\left( D^{2}u+D^{2}c(x,T(x))\right) -\ln\det(-D\bar{Dc(x,T(x)))}\\
-\ln(t\mu(x)+(1-t)m(x))+\ln f^{-1}\left( Q(t,x,T(x))\right) \label{continuity}%
\end{multline}
where
\[
Q(t,x,T(x))=
\]
\[
Q(t,x,T(x))=-u(x)-c(x,T(x))-t\int\phi(T(x),T(z))d\mu(z)-tV(T(x))\}.
\]
At time $t=0,~\ $a solution is given by $u\equiv0:$ \ This maps the measure
$m$ to itself via the identity mapping. \ The interval $\mathcal{I}$ of $t$
for which a solution exists is nonempty. \ Notice that the form of the
equation (\ref{continuity}) is the same form as (\ref{lnequation}) up to a
scale of the functions $\phi$ and $V$ and a change of measure, so the
estimates from the previous section all hold. From the theory of Krylov and
Evans one can obtain $C^{2,\alpha}$ estimates. Thus $\mathcal{I}$ is closed.
\ \ Lemma \ref{lem:invertible L} with these conditions gives openess, noting
that on the sphere, a Laplacian has index zero, and that the linearized
operator which has the same principle symbol has index zero as well. \ 

\begin{rem*}
For uniqueness, the standard PDE trick does not work immediately, even under
assumptions such as those in the theorem. One may be tempted by the standard
argument \cite[Theorem 17.1]{GT} to obtain a contradiction. However, the
intermediate linearized operator will have the additional $\nabla F$ term that
arises in (\ref{full linearized operator}) because combinations of $u$ and $v$
are not solutions. Our proof of invertibility fails for these, so we have no
reason to expect the proof would hold after being integrated. Uniqueness may be more easily obtained from geometric consideration as in \cite[section 4]{BC},
see also \cite[Chapters 15, 16]{MR2459454Villanibook}.

However, if the integral term is not present, we can use the argument
\cite[Theorem 17.1]{GT}, making the important note that on the sphere, the set
of c-convex function is convex \cite[Theorem 3.2]{FKMScreening}. In this case
invertibililty of the linearized operator follows easily from standard maximum
principle arguments.
\end{rem*}

\bibliographystyle{plain}
\bibliography{BCAug2012}

\begin{thebibliography}{10}

\bibitem{BC}
Adrien Blanchet and Guillaume Carlier.
\newblock Optimal transport and cournot-nash equilibria.
\newblock {\em Preprint}.

\bibitem{MR2232207}
Philippe Delano{\"e} and Gr{\'e}goire Loeper.
\newblock Gradient estimates for potentials of invertible gradient-mappings on
  the sphere.
\newblock {\em Calc. Var. Partial Differential Equations}, 26(3):297--311,
  2006.

\bibitem{FKMScreening}
Alessio Figalli, Young-Heon Kim, and Robert~J. McCann.
\newblock When is multidimensional screening a convex program?
\newblock {\em J. Econom. Theory}, 146(2):454--478, 2011.

\bibitem{GT}
David Gilbarg and Neil~S. Trudinger.
\newblock {\em Elliptic partial differential equations of second order}.
\newblock Classics in Mathematics. Springer-Verlag, Berlin, 2001.
\newblock Reprint of the 1998 edition.

\bibitem{MR2654086}
Young-Heon Kim and Robert~J. McCann.
\newblock Continuity, curvature, and the general covariance of optimal
  transportation.
\newblock {\em J. Eur. Math. Soc. (JEMS)}, 12(4):1009--1040, 2010.

\bibitem{Kim03102011}
Young-Heon Kim, Jeffrey Streets, and Micah Warren.
\newblock Parabolic optimal transport equations on manifolds.
\newblock {\em International Mathematics Research Notices}, 2011.

\bibitem{MR2644782}
Jiakun Liu and Neil~S. Trudinger.
\newblock On {P}ogorelov estimates for {M}onge-{A}mp\`ere type equations.
\newblock {\em Discrete Contin. Dyn. Syst.}, 28(3):1121--1135, 2010.

\bibitem{MR2506751}
Gr{\'e}goire Loeper.
\newblock On the regularity of solutions of optimal transportation problems.
\newblock {\em Acta Math.}, 202(2):241--283, 2009.

\bibitem{MR2188047}
Xi-Nan Ma, Neil~S. Trudinger, and Xu-Jia Wang.
\newblock Regularity of potential functions of the optimal transportation
  problem.
\newblock {\em Arch. Ration. Mech. Anal.}, 177(2):151--183, 2005.

\bibitem{MR2459454Villanibook}
C{\'e}dric Villani.
\newblock {\em Optimal transport}, volume 338 of {\em Grundlehren der
  Mathematischen Wissenschaften [Fundamental Principles of Mathematical
  Sciences]}.
\newblock Springer-Verlag, Berlin, 2009.
\newblock Old and new.

\end{thebibliography}

\end{document}